\documentclass[pdftex]{svmult}

\usepackage{mathptmx}       % selects Times Roman as basic font
\usepackage{helvet}         % selects Helvetica as sans-serif font
\usepackage{courier}        % selects Courier as typewriter font
\usepackage{type1cm}        % activate if the above 3 fonts are
\usepackage{makeidx}         % allows index generation
\usepackage{graphicx}        % standard LaTeX graphics tool
\usepackage{multicol}        % used for the two-column index
\usepackage[bottom]{footmisc}% places footnotes at page bottom
\usepackage{amsmath}
\usepackage{amsfonts}
\usepackage{amssymb}
\usepackage[T1]{fontenc}
\usepackage[utf8]{inputenc}
\usepackage{graphicx}
\usepackage[usenames, dvipsnames]{color}
\usepackage{tikz}
\usetikzlibrary{plotmarks}
\usepackage{pgfplots}
\pgfplotsset{compat=newest}
\newlength\figurewidth

\usepackage[activate={true,nocompatibility},spacing,kerning]{microtype}
\usepackage[pass]{geometry}

\definecolor{gray}{rgb}{0.9,0.9,0.9}
\definecolor{windowsgray}{rgb}{0.8,0.8,0.8}
\definecolor{myblue}{rgb}{0.0,0.475,0.758}
\definecolor{darkgreen}{rgb}{0,0.6,0}
\definecolor{darkred}{rgb}{0.8,0,0}
\definecolor{darkblue}{rgb}{0,0,0.8}
\definecolor{middlegray}{rgb}{0.7,0.7,0.7}

\makeindex             % used for the subject index
\renewcommand{\leq}{\leqslant}
\renewcommand{\geq}{\geqslant}

\newcommand{\N}{\ensuremath{\mathbb{N}}}
\newcommand{\C}{\ensuremath{\mathbb{C}}}
\newcommand{\Z}{\ensuremath{\mathbb{Z}}}

\let\Re\relax
\let\Im\relax
\DeclareMathOperator\Re{Re}
\DeclareMathOperator\Im{Im}
\DeclareMathOperator\GL{GL}
\DeclareMathOperator\Log{Log}
\DeclareMathOperator\ind{ind}
\DeclareMathOperator\coker{coker}
\DeclareMathOperator\1{id}
\DeclareMathOperator\res{res}

\usepackage{listings}
\definecolor{mygreen}{rgb}{0,0.6,0}
\definecolor{mygray}{rgb}{0.5,0.5,0.5}
\definecolor{lightgray}{rgb}{0.925,0.925,0.925}
\definecolor{mymauve}{rgb}{0.58,0,0.82}
\lstdefinelanguage{Julia}%
  {morekeywords={abstract,break,case,catch,const,continue,do,else,elseif,%
      end,export,false,for,function,immutable,import,importall,if,in,%
      macro,module,otherwise,quote,return,switch,true,try,type,typealias,%
      using,while,norm,Fun,cauchy,Circle,Cauchy,@time,exp,lfact},%
   sensitive=true,%
   alsoother={$},%
   morecomment=[l]\#,%
   morecomment=[n]{\#=}{=\#},%
   morestring=[s]{"}{"},%
   morestring=[m]{'}{'},%
}[keywords,comments,strings]%
\begin{document}

%opening
\title*{Numerical Methods for the Discrete Map \boldmath$Z^a$\unboldmath}

\author{Folkmar Bornemann, Alexander Its, Sheehan Olver, and Georg Wechslberger}

\maketitle

\vspace*{-2.75cm}
{\bf Abstract}\quad As a basic example in nonlinear theories of discrete complex analysis, 
we explore various numerical methods for the accurate evaluation
of the discrete map $Z^a$ introduced by Agafonov and Bobenko. The methods are  based either
on a discrete Painlevé equation or on the Riemann--Hilbert method. In the latter case, the
 underlying structure of a triangular Riemann--Hilbert problem with a non-triangular solution
 requires special care in the numerical approach.
  Complexity and numerical stability are discussed, the results are illustrated by numerical examples.

\section{Introduction}

Following the famous ideas of Thurston's for a \emph{nonlinear} theory of discrete complex analysis
based on circle packings, Bobenko and Pinkall \cite{BP96} defined a {\em discrete conformal} map as a 
complex valued function $f:\Z^2 \subset \C \to \C$ satisfying
\begin{equation}\label{eq1}
\frac{(f_{n,m} - f_{n+1,m})(f_{n+1,m+1} - f_{n,m+1})}
{(f_{n+1,m} - f_{n+1,m+1})(f_{n,m+1} - f_{n,m})} = -1.
\end{equation}
That is, the cross ratio on each elementary quadrilateral (fundamental cell) of the lattice $\Z^2$ is $-1$; infinitesimally, 
this property characterizes conformal maps among the smooth ones. A discrete conformal map $f_{n,m}$ is
called an {\em immersion} if the interiors of adjacent elementary quadrilaterals are disjoint.

A central problem in discrete complex analysis is to find discrete conformal analogues of classical holomorphic
functions that are immersions;  simply evolving just the boundary values of the classical
function by \eqref{eq1} would not work \cite{MR1676682}. To solve this problem, Bobenko \cite{MR1705222} suggested 
to augment \eqref{eq1} by another equation:
using methods from the theory of integrable systems it can be shown that the non-autonomous system of constraints
\begin{equation}\label{eq2}
a f_{n,m} =2n\frac{(f_{n+1,m} - f_{n,m})(f_{n,m} - f_{n-1,m})}
{(f_{n+1,m} - f_{n-1,m})}
+ 
2m\frac{(f_{n,m+1} - f_{n,m})(f_{n,m} - f_{n,m-1})}
{(f_{n,m+1} - f_{n,m-1})},
\end{equation}
obtained as an integrable discretization of the differential equation
\[
a f = x f_x + y f_y = z f_z
\]
that would define $f(z)=z^a$ up to scaling, is compatible with \eqref{eq1}. Agafonov and Bobenko \cite{AB00}
proved that, for $0<a<2$, the 
system \eqref{eq1} and \eqref{eq2} of recursions, applied to the three
initial values
\begin{equation}\label{eq3}
f_{0,0} = 0,\quad f_{1,0} = 1, \quad f_{0,1} = e^{ia\pi/2},
\end{equation}
defines a unique discrete conformal map $Z^a_{n,m} = f_{n,m}$ that is an immersion \cite[Thm.~1]{AB00}.

\begin{figure}[tbp]
\centering
\begin{minipage}{0.425\textwidth}
\setlength\figurewidth{\textwidth}
\input{Images/F20}\\*[1mm]
~
\end{minipage}\hfil
\begin{minipage}{0.475\textwidth}
\includegraphics[width=\textwidth]{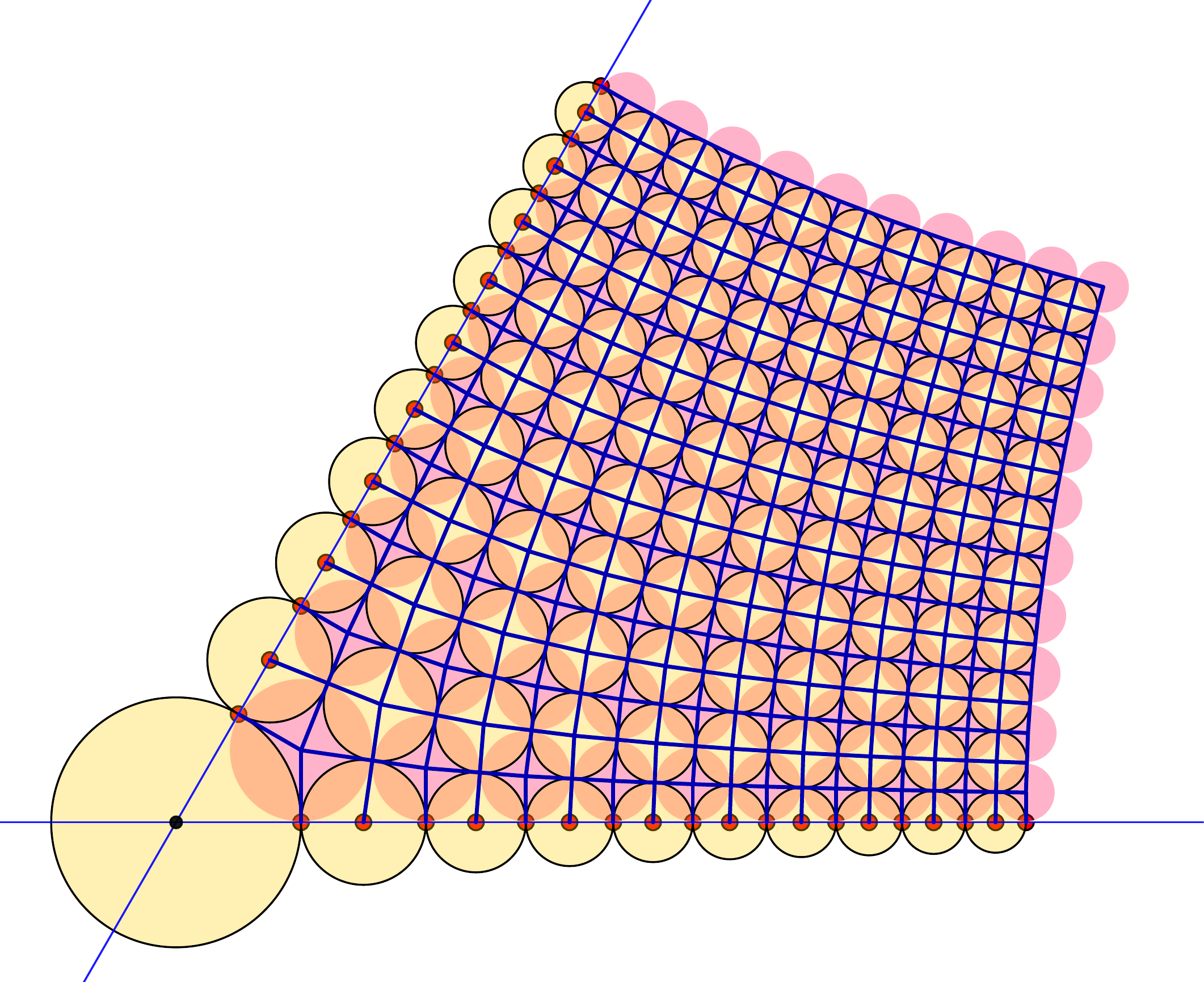}
\end{minipage}
  \caption{Left: red dots are the discrete $Z^{2/3}$ for $0\leq n,m \leq 19$; blue circles are the asymptotics
  given by \eqref{eq:asympt}. Right: the Schramm circle pattern of the discrete $Z^{2/3}$ [courtesy of
  J. Richter-Gebert].
  }\label{fig1}
\end{figure}

Moreover, they  showed \cite[Sect.~3]{AB00} that this discrete conformal map $Z^a$ determines a circle pattern of Schramm
type, i.e., an orthogonal circle pattern with the combinatorics of the square grid, see Fig.~\ref{fig1}. They conjectured,
recently proved by Bobenko and Its \cite{BI15} using the Riemann--Hilbert method, that asymptotically
\begin{subequations}\label{eq:asympt}
\begin{equation}
Z^a_{n,m} = c_a\left(\frac{n+im}{2}\right)^a
 \left( 1 + O\left(\frac{1}{n^2 + m^2}\right)\right) \qquad (n^2 + m^2 \to \infty)
\end{equation}
with the constant
\begin{equation}
c_a = \frac{\Gamma\left(1 - \frac{a}{2}\right)}{\Gamma\left(1 + \frac{a}{2}\right)}.
\end{equation}
\end{subequations}
For $0<a\leq 1$, as exemplified in Fig.~\ref{fig1}, this asymptotics is  already accurate to plotting accuracy
for all but the very smallest values of $n$ and $m$. If $a\to 2$, however, it requires increasingly larger values of $n$ and $m$
to become accurate.

In this work we study the stable and accurate \emph{numerical} calculation of $Z^a$; to the best of our knowledge for the first time in the literature. This is an interesting 
mathematical problem in itself,
but the underlying methods should be applicable to a large set of similar discrete integrable systems. Now, the basic difficulty
is that the evolution of the discrete dynamical system \eqref{eq1} and \eqref{eq2}, 
starting from the initial values~\eqref{eq3}, is numerically highly unstable, see Fig.~\ref{fig2}.\footnote{All numerical
calculations are done in hardware arithmetic using double precision.}

\begin{figure}[tbp]
\centering
\setlength\figurewidth{0.7\textwidth}
\input{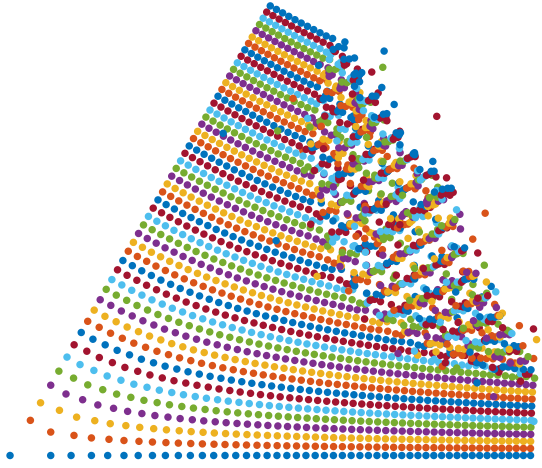}
\caption{Numerical discrete $Z^{2/3}_{n,m}$ ($0\leq n,m\leq 49$): recursing from the initial values \eqref{eq3} by a straightforward
application of the system \eqref{eq1} and \eqref{eq2} quickly develops numerical instabilities. The color
cycles with the coordinate $m$.}\label{fig2}
\end{figure}

The support of the stencils of \eqref{eq1} and \eqref{eq2} has the form of a square and a five-point cross in the lattice 
$\Z^2$, that is,
\[
\begin{matrix}
f_{n,m+1} & f_{n+1,m+1}\\*[1mm]
f_{n,m} & f_{n+1,m}
\end{matrix}\qquad\text{and}\qquad
\begin{matrix}
 & f_{n,m+1}\\*[1mm]
f_{n-1,m} & f_{n,m} & f_{n+1,m}\\*[1mm]
&f_{n,m-1}
\end{matrix},
\]
with the latter reducing to be dimensional along the boundary of $\Z^2_+$, namely
\[
\begin{matrix}
f_{0,m+1}\\*[1mm]
f_{0,m}\\*[1mm]
f_{0,m-1}
\end{matrix}, \qquad\text{resp.}\qquad
\begin{matrix}
f_{n-1,0} & f_{n,0} & f_{n+1,0},
\end{matrix}
\]
if $n=0$ or $m=0$. Thus, the forward evolution can be organized as follows. If $f_{n,m}$ is known for
$0\leq n,m \leq N$ the upper index $N$ is increased to $N+1$ according to:
\begin{center}
\begin{minipage}{0.8\textwidth}
\begin{enumerate}
\item use \eqref{eq2} to compute the {\em boundary} values $f_{N+1,0}$ and $f_{0,N+1}$;
\item use \eqref{eq1} to compute the {\em row} $f_{N+1,m}$, $1\leq m \leq N$;
\item use \eqref{eq1} to compute the {\em column} $f_{n,N+1}$, $1\leq n \leq N$;
\item use \eqref{eq1} to compute the {\em diagonal} value $f_{N+1,N+1}$.
\end{enumerate}
\end{minipage}
\end{center}
It is this algorithm that gives the unstable calculation shown in Fig.~\ref{fig2}. Alternatively, one could use \eqref{eq2} to calculate the row values $f_{N+1,m}$, $1\leq m \leq N-1$,
and column values  $f_{n,N+1}$, $1\leq n \leq N-1$, up to the first sub- and
superdiagonal (note that these calculations do \emph{not} depend on order within the rows and columns). The missing values are then completed by using \eqref{eq1}. However, this alternative
forward evolution gives a result that is visually indistinguishable from Fig.~\ref{fig2}. 

As can be seen from Fig.~\ref{fig2}, the numerical instability starts spreading from the diagonal elements $f_{n,n}$. 
In fact, there is an initial exponential growth of numerical errors to be found in the diagonal entries, see
Fig.~\ref{fig3}. Such a numerical instability of an evolution is the direct consequence of the instability of the
underlying dynamical system, that is, of positive Lyapunov exponents.

\begin{figure}[tbp]
\centering
\setlength\figurewidth{0.575\textwidth}
% This file was created by matlab2tikz v0.5.0 running on MATLAB 8.4.
%Copyright (c) 2008--2014, Nico Schlömer <nico.schloemer@gmail.com>
%All rights reserved.
%Minimal pgfplots version: 1.3
%
%
% defining custom colors
\definecolor{mycolor1}{rgb}{0.00000,0.44700,0.74100}%
\definecolor{mycolor2}{rgb}{0.85000,0.32500,0.09800}%
\definecolor{mycolor3}{rgb}{0.92900,0.69400,0.12500}%
\begin{tikzpicture}

\begin{axis}[%
width=0.950920245398773\figurewidth,
height=0.75\figurewidth,
at={(0\figurewidth,0\figurewidth)},
scale only axis,
separate axis lines,
every outer x axis line/.append style={white!15!black},
every x tick label/.append style={font=\color{white!15!black}},
xmin=0,
xmax=50,
xlabel={$n$},
xmajorgrids,
every outer y axis line/.append style={white!15!black},
every y tick label/.append style={font=\color{white!15!black}},
ymode=log,
ymin=1e-16,
ymax=100,
yminorticks=true,
ylabel={error},
ymajorgrids,
yminorgrids
]
\addplot [color=mycolor1,line width=0.8pt,only marks,mark=*,mark options={solid},forget plot]
  table[row sep=crcr]{%
2	1.11022302462516e-16\\
3	7.7715611723761e-16\\
4	1.60118641699469e-15\\
5	7.58859974442808e-15\\
6	2.69141510402835e-14\\
7	1.28168985305061e-13\\
8	7.05443083338607e-13\\
9	4.01707458109793e-12\\
10	2.26515885533321e-11\\
11	1.27169102447025e-10\\
12	7.11415993308925e-10\\
13	3.97407673261213e-09\\
14	2.21961230460869e-08\\
15	1.24045485986177e-07\\
16	6.9395712480086e-07\\
17	3.88708517603187e-06\\
18	2.18015646031847e-05\\
19	0.000122431145733303\\
20	0.000688015612859279\\
21	0.00385712553921828\\
22	0.0211915183519683\\
23	0.102925849227913\\
24	0.292487344486066\\
25	0.398995546860738\\
26	0.44166555356837\\
27	0.565910722898638\\
28	0.764411614547993\\
29	0.843505415515225\\
30	0.9321123919951\\
31	1.07478227918614\\
32	1.2584346750076\\
33	1.42137807508256\\
34	1.48437938030345\\
35	1.57593584633722\\
36	1.77302113182222\\
37	1.88678766270227\\
38	1.97957942757555\\
39	2.09465842228319\\
40	2.27349617275413\\
41	2.41946237140379\\
42	2.4745874256652\\
43	2.58336659852967\\
44	2.78157637964196\\
45	2.90221574173493\\
46	2.97896133165648\\
47	3.08371736562182\\
48	3.3014333114154\\
49	3.39176084141234\\
50	3.44471781019914\\
};
\addplot [color=mycolor2,line width=0.8pt,only marks,mark=*,mark options={solid},forget plot]
  table[row sep=crcr]{%
2	4.00296604248672e-16\\
3	1.73422380365255e-15\\
4	8.24634920558061e-15\\
5	4.2169333854586e-14\\
6	2.21412487868285e-13\\
7	1.18416311300408e-12\\
8	6.41434845130347e-12\\
9	3.50707999144994e-11\\
10	1.93143228868507e-10\\
11	1.06983845824283e-09\\
12	5.95386997095598e-09\\
13	3.32642485480112e-08\\
14	1.86461322613817e-07\\
15	1.0481544777179e-06\\
16	5.90638889464285e-06\\
17	3.33537818796904e-05\\
18	0.000188705553078116\\
19	0.00106942700944195\\
20	0.00606983261987984\\
21	0.034498496672338\\
22	0.195630863025333\\
23	0.983147730908515\\
24	1.92369925038217\\
25	2.00333742843315\\
26	1.88175211032564\\
27	0.837757441578744\\
28	0.00558496147435826\\
29	0.797462612411627\\
30	1.86972737389947\\
31	2.00330890247277\\
32	1.68135146223126\\
33	0.471819320988728\\
34	0.353507326930769\\
35	1.54343491090764\\
36	2.00218218226826\\
37	1.84650408867109\\
38	0.73968515579186\\
39	0.236213320371232\\
40	1.42067451232662\\
41	1.9971658356739\\
42	1.84842006698433\\
43	0.743434930411291\\
44	0.295007389735192\\
45	1.51409180756564\\
46	2.0058219471549\\
47	1.75346982165874\\
48	0.557826524882148\\
49	0.509678182903123\\
50	1.71531314787688\\
};
\addplot [color=mycolor3,line width=0.8pt,only marks,mark=*,mark options={solid},forget plot]
  table[row sep=crcr]{%
4	2.22044604925031e-16\\
5	4.44089209850063e-16\\
6	3.10862446895044e-15\\
7	1.68753899743024e-14\\
8	9.14823772291129e-14\\
9	5.01154673315796e-13\\
10	2.76001443921814e-12\\
11	1.52879930936933e-11\\
12	8.50812753583341e-11\\
13	4.75348871376013e-10\\
14	2.66454813768746e-09\\
15	1.49782173330948e-08\\
16	8.44028040791756e-08\\
17	4.766284114055e-07\\
18	2.69661874874316e-06\\
19	1.52822079173554e-05\\
20	8.67382208145084e-05\\
21	0.000492919882341125\\
22	0.00278260842782219\\
23	0.012294583442984\\
24	0.0085977185890822\\
25	0.003582302946739\\
26	0.00995108744688111\\
27	0.0116776503204359\\
28	0.00488422127499577\\
29	0.0112570960702774\\
30	0.0111500823454327\\
31	0.00598135421790214\\
32	0.0139863278667578\\
33	0.00912627174505287\\
34	0.00805421472102985\\
35	0.0151152386162916\\
36	0.00752517486087334\\
37	0.0121726141286096\\
38	0.0120629638686995\\
39	0.00834464423392611\\
40	0.0155582890501429\\
41	0.00879194947096318\\
42	0.0126418745198125\\
43	0.012597277295378\\
44	0.0094666511141237\\
45	0.0157736310976184\\
46	0.00927908847350056\\
47	0.0143755642173253\\
48	0.0117878870066137\\
49	0.0113408717070878\\
50	0.0151111388312384\\
};
\end{axis}
\end{tikzpicture}%
\caption{Numerical error of the diagonal values $f_{n,n}$ from Fig.~\protect\ref{fig2} (blue),
of the $x_n$ as in~\protect\eqref{eq:xn} and computed by forward evolution of the discrete Painlevé II equation (red), and of the corresponding invariant $|x_n|=1$ (yellow); $a=2/3$. 
They share the same rate of initial exponential growth.}\label{fig3}
\end{figure}
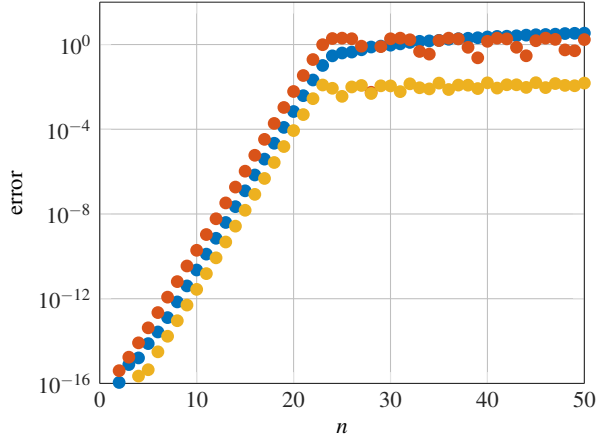

As a remedy we suggest two different approaches to calculating $Z^a_{n,m}$. In Section~\ref{sect:stable}
we stabilize the calculation of the diagonal values by solving a boundary value problem for an underlying discrete
Painlevé II equation and in Sections~\ref{sect:RHP}--\ref{sect:ny} we explore numerical methods based on the Riemann--Hilbert method. The latter reveals an interesting structure (Section~\ref{sect:triang}): the Riemann--Hilbert problem has triangular data but a non-triangular solution; the operator equation can thus be written as a uniquely solvable block triangular system where the {\em infinite-dimensional} diagonal operators are {\em not} invertible. We discuss two different ways to prevent this particular structure
from hurting {\em finite-dimensional} numerical schemes:  a coefficient-based spectral method with infinite-dimensional linear algebra in Section~\ref{sect:spectral} and a modified Nyström method based on least squares in Section~\ref{sect:ny}.

\section{Discrete Painlevé II Separatrix as a Boundary Value Problem}\label{sect:stable}

Since the source of the numerical instability of the direct evolution of the discrete dynamical system \eqref{eq1} and \eqref{eq2}
is found in the  diagonal elements $f_{n,n}$, we first express the $f_{n,n}$ directly in terms of a one-dimensional
three-term 
recursion and then study its stable numerical evaluation. To begin with, Agafonov and Bobenko \cite[Prop.~3]{AB00} proved that the geometric quantities
\begin{equation}\label{eq:xn}
x_n^2 = \frac{f_{n,n+1}-f_{n,n}}{f_{n+1,n}-f_{n,n}},\qquad \arg x_n \in (0,\pi/2),
\end{equation}
have invariant magnitude $|x_n|=1$ (see the circle packing in Fig.~\ref{fig1}) and that they satisfy the following form
of the discrete Painlevé II equation
\begin{equation}
(n+1)(x_n^2-1)\left(\frac{x_{n+1}-ix_n}{i+x_nx_{n+1}}\right)- 
 n(x_n^2+1)\left(\frac{x_{n-1}+ix_n}{i+x_{n-1}x_n}\right) =
a x_n,
\label{eq:painleve}
\end{equation}
with initial value $x_0 = e^{ia \pi /4}$. Note that for $n=0$ this nonlinear three-term recurrence degenerates and gives
the missing second initial value, namely
\begin{equation}\label{eq:x1}
x_1=\frac{x_0(x_0^2+a -1)}{i((a -1)x_0^2+1)}.
\end{equation}
Reversely, given the solution $x_n$ of this equation, the diagonal elements
$f_{n,n}$ can be calculated according to the simple recursion \cite[p.~176]{AB00}
\begin{equation}\label{eq:diagonal}
u_n = \frac{r_n}{\Re x_n}, \quad r_{n+1}=u_n \cdot \Im x_n, \quad g_{n+1} = g_n + u_n, 
\quad f_{n+1,n+1} = g_{n+1} e^{ia\pi/4},
\end{equation}
with inital values $g_0 = 0$, $r_0=1$ (note that $u_n$, $r_n$, $g_n$ are all positive); the sub- and superdiagonal elements $f_{n+1,n}$ and $f_{n,n+1}$ are obtained
from \eqref{eq1} and \eqref{eq:xn} by
\begin{subequations}\label{eq:subdiagonal}
\begin{align}
f_{n+1,n} &= \frac{(x_n^2-1)f_{n,n} + (x_n^2+1)f_{n+1,n+1}}{2x_n^2},\\*[2mm]
f_{n,n+1} &= \frac{(1-x_n^2)f_{n,n} + (1+x_n^2)f_{n+1,n+1}}{2}.
\end{align}
\end{subequations}
However, given that $x_n$ is a \emph{separatrix} solution of the discrete Painlevé II equation \cite[p.~167]{AB00},
we expect that a forward evolution of \eqref{eq:painleve}, starting with the initial
values $x_0$ and $x_1$, suffers from exactly the same instability as the calculation of the diagonal values $f_{n,n}$
by evolving \eqref{eq1} and \eqref{eq2}. Fig.~\ref{fig3} shows that this is indeed the case, exhibiting
the same initial exponential growth rate; it also shows that the deviation of the calculated values of $|x_n|$ 
from its invariant value $1$ can serve as an explicitly computable error indicator.

In the continuous case of the Hastings--McLeod solution of Painlevé II, which also constitutes a separatrix, Bornemann 
\cite[Sect.~3.2]{B10}
suggested to address such problems by solving an asymptotic two-point {\em boundary} value problem instead of the originally given
evolution problem. To this end, one has to solve the {\em connection problem} first, that is, one has to establish the
asymptotics of $x_n$ as $n\to \infty$. By inserting the known asymptotics \eqref{eq:asympt} of $Z_{n,m}^a$
into the defining equation \eqref{eq:xn}, we obtain
\[
x_n = e^{i\pi/4}(1+O(n^{-1}))\qquad (n\to\infty).
\]
Since in actual numerical calculations we need accurate approximations already for moderately large $n$, we 
match the coefficients of an expansion in terms of $n^{-1}$ to the discrete Painlevé II equation \eqref{eq:painleve}
and get, as $n\to\infty$,
{\small
\begin{multline*}%\label{eq:asympt_xn}
x_n = 
e^{\frac{i \pi }{4}} 
\bigg(1+\frac{i (a -1)}{2
   n} + \frac{-a ^2+(2-2 i) a
   -(1-2 i)}{8 n^2} 
   - \frac{i
   \left(a ^3-(3-2 i) a ^2-(1+4
   i) a +(3+2 i)\right)}{16
   n^3}\\*[2mm]
   +\frac{3 a ^4-(12-12 i) a
   ^3-(2+36 i) a ^2+(28+4 i) a
   -(17-20 i)}{128 n^4}\\*[2mm]
   +\frac{i \left(3
   a ^5-(15-12 i) a ^4-(30+48 i)
   a ^3+(150+24 i) a ^2-(5-48 i)
   a -(103+36 i)\right)}{256
   n^5}\\*[2mm] + O(n^{-6})\bigg).
\end{multline*}}%
We denote the r.h.s. of this asymptotic formula, without the $O(n^{-6})$ term, by $x_{n,6}$.
\begin{figure}[tbp]
\centering
\setlength\figurewidth{0.7\textwidth}
\vspace*{0.41cm}
\input{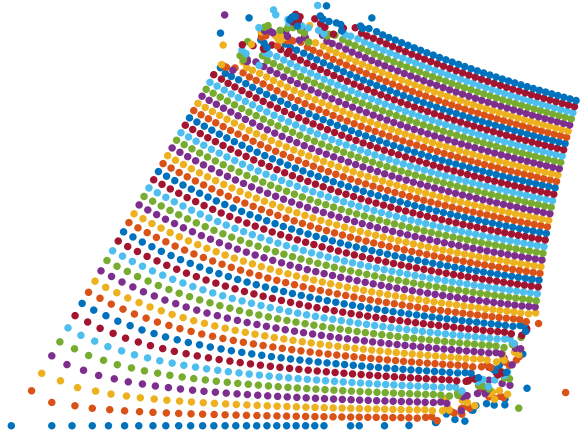}
\caption{Numerical discrete $Z^{2/3}_{n,m}$ ($0\leq n,m\leq 49$): evolving from accurate values of $f_{n,n}$,
$f_{n+1,n}$, $f_{n,n+1}$ close to the diagonal back to the boundary by using the cross-ratio relations \eqref{eq1} develops numerical instabilities. The color
cycles with the coordinate $m$.}\label{fig4}
\end{figure}

Next, using Newton's method, we solve the nonlinear system of $N+1$ equations in $N+1$ unknowns $x_0,\ldots,x_N$
given by the discrete Painlevé equation \eqref{eq:painleve}
for $1\leq n \leq N-1$ and the two boundary conditions
\[
x_1=\frac{x_0(x_0^2+a -1)}{i((a -1)x_0^2+1)},\qquad x_N = x_{N,6}.
\]
Note that the value $x_0 = e^{ia\pi/4}$ is not explicitly used and must be obtained as output of the Newton solve,
that is, it can be used as an measure of success. We choose $N$ large enough that $|x_{N,6}| \doteq 1 $ up to machine precision
(about $N\approx 300$ uniformly in $a$). Then, using the excellent initial guesses (for the accuracy of the
asymptotics cf. the left panel of Fig.~\ref{fig1})
\[
x_0^{(0)} = e^{ia\pi/4},\qquad x_n^{(0)} = x_{n,6}\quad (1\leq n \leq N),
\]
Newton's method will converge in about just 10 iterations to machine precision yielding a numerical solution that
satisfies the invariant $|x_n|=1$ also up to machine precision. Since the Jacobian of a nonlinear system stemming
from a three-term recurrence is {\em tridiagonal}, each Newton step has an operation count of order $O(N)$. Hence,
the overall complexity of accurately calculating the values $x_n$, $0\leq n \leq N$, is of optimal order $O(N)$.

Finally, having accurate values of $x_n$ at hand, and therefore by \eqref{eq:diagonal} and \eqref{eq:subdiagonal} also those of
$f_{n,n}$, $f_{n+1,n}$ and $f_{n,n+1}$, one can calculate the missing values of $f_{n,m}$ row- and column-wise,
starting from the second sup- and superdiagonal and evolving to the boundary, either by evolving the cross-ratio
relations \eqref{eq1} or by evolving the discrete differential equation \eqref{eq2}. It turns out that the first option 
develops numerical
instabilities spreading from the boundary, see Fig.~\ref{fig4}, whereas the second option is, for a wide range
of the parameter $a$, numerically observed to be perfectly stable,
see Fig.~\ref{fig5}. Note that this stable algorithm only differs from the alternative direct evolution discussed
in the introduction in how the values  close to the diagonal, that is $f_{n,n}$,
$f_{n+1,n}$ and $f_{n,n+1}$, are computed. 

The total complexity of this stable numerical calculation of the array $f_{n,m}$ with $0\leq n,m \leq N$ 
is of optimal order $O(N^2)$. 

\begin{figure}[tbp]
\centering
\setlength\figurewidth{0.7\textwidth}
\input{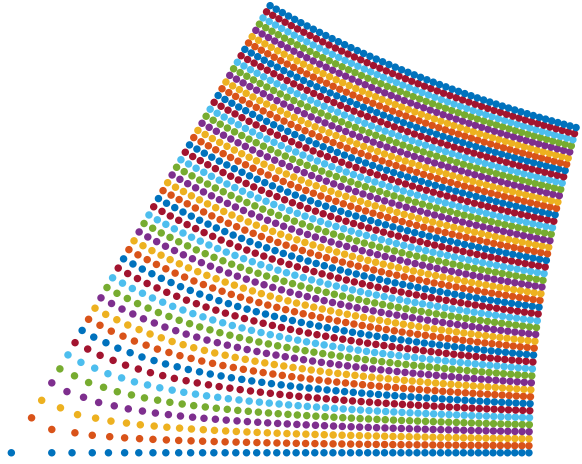}
\caption{Numerical discrete $Z^{2/3}_{n,m}$ ($0\leq n,m\leq 49$): recursing from accurate values of $f_{n,n}$,
$f_{n+1,n}$, $f_{n,n+1}$ close to the diagonal back to the boundary by using the discrete differential equation \eqref{eq2} is perfectly stable. The color
cycles with the coordinate $m$.
  }\label{fig5}
\end{figure}

\section{The Riemann--Hilbert Method}\label{sect:RHP}

\begin{figure}[tbp]
\centering
   \begin{minipage}[t][2cm]{4cm}
	   \vspace*{3mm} % without vspace the two minipages are not aligned
	   \begin{tikzpicture}
	    	[rhcontour/.style={->,>=to,ultra thick}]
	       \draw [rhcontour] (-1,0) arc[radius=0.75, start angle=90, end angle=450]; 
	       \draw [rhcontour]  (1,0) arc[radius=0.75, start angle=90, end angle=450]; 
	       \draw [rhcontour,white]  (0,1) arc[radius=2, start angle=90, end angle=450];  
	   \end{tikzpicture}
   \end{minipage}
\hfil  
      \begin{minipage}[t]{4cm}
	   \vspace{0.1pt} % without vspace the two minipages are not aligned
	   \begin{tikzpicture}
	    	[rhcontour/.style={->,>=to,ultra thick}]
	       \draw [rhcontour] (-1,0) arc[radius=0.75, start angle=90, end angle=450]; 
	       \draw [rhcontour]  (1,0) arc[radius=0.75, start angle=90, end angle=450]; 
	       \draw [rhcontour,darkred]  (0,1.25) arc[radius=2, start angle=90, end angle=450]; 
	   \end{tikzpicture}
   \end{minipage}
\caption{Contours for the $X$-RHP \cite[p.~15]{BI15}. Left: two non-intersecting circles $\Gamma_1$ centered at $\pm 1$ (black); right: additional circle $\Gamma_2$ centered at $0$ (red) for standard normalization at $\infty$.
  }\label{fig6}
\end{figure}
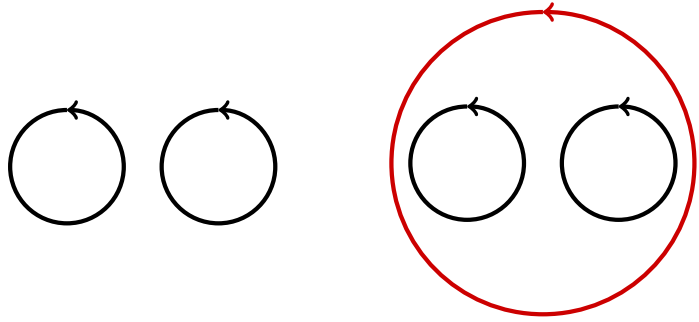

Based on the integrability of the system \eqref{eq1} and \eqref{eq2}, by identifying \eqref{eq1} as the 
compatibility condition of a Lax pair of linear difference equations \cite{BP96} and  by using isomonodromy, Bobenko and Its \cite[p.~15]{BI15} expressed the $Z^a$ map
in terms of the following Riemann--Hilbert problem (which is a slightly transformed and transposed version of the $X$-RHP by these authors): Let $\Gamma_1$ be the oriented contour built of
two non-intersecting circles in the complex plane centered at $z=\pm 1$ (see Fig.~\ref{fig6} left), the holomorphic
function $X: \C\setminus\Gamma_1 \to \GL(2)$ satisfies the jump condition
\begin{subequations}\label{eq:XRHP}
\begin{equation}
X_+(\zeta) = G_1(\zeta) X_-(\zeta)\qquad (\zeta \in \Gamma_1)
\end{equation}
with the jump matrix\footnote{To make $G_1$ holomorphic in the vicinity of $\Gamma_1$ we place the branch-cut of $\zeta^{-a/2}$
at the negative imaginary axis, that is, we take, using the principal branch $\Log$ of the logarithm,
\[
e^{ia \pi/2} \zeta^{-a/2} = e^{ia \pi/4} e^{-\frac{a}{2} \Log(\zeta/i)}.
\]
}
\begin{equation}\label{eq:G1}
G_1(\zeta) =
\begin{pmatrix} 
1 & 0 \\*[2mm]
e^{ia \pi/2} \zeta^{-a/2} (\zeta-1)^{-m} (\zeta+1)^{-n}\; & 1 
\end{pmatrix}
\end{equation}
subject to the following normalization
\begin{equation}\label{eq:norm_non_standard}
X(z) = 
\begin{pmatrix}
z^{\frac{m+n}{2}} & 0 \\*[1mm]
0 & z^{-\frac{m+n}{2}}
\end{pmatrix}(I+O(z^{-1}))\qquad (z\to \infty).
\end{equation}
Here, we restrict ourselves to values of $n$ and $m$ having the same parity such that $(m+n)/2$ is an integer.
The discrete $Z^a$ map is now given by the values $f_{n,m}$ extracted from an $LU$-decomposition at $z=0$, namely, 
\[
X(0) = 
\begin{pmatrix}
1 & 0\\*[1mm]
(-1)^{m+1}f_{n,m}\; & 1 
\end{pmatrix}
\begin{pmatrix}
\;\bullet\; & \;\bullet\;\\*[1mm]
0 & \;\bullet\; 
\end{pmatrix},
\]
that is,
\[
f_{n,m} = (-1)^{m+1} \frac{X_{21}(0)}{X_{11}(0)}.
\]
Subsequently, using the Deift--Zhou nonlinear steepest decent method,
Bobenko and Its  \cite{BI15} transform this $X$-RHP to a series of Riemann--Hilbert problems that are more suitable for
asymptotic analysis. The last one of this series before introducing a global parametrix,\footnote{Though 
the parametrix leads to a near-identity RHP, the actually computation of the para\-metrix would require solving a problem 
that is, numerically, of similar difficulty as the $S$-RHP itself.} the $S$-RHP \cite[pp. 24--27]{BI15},
is based on the contour shown in the left part of Fig.~\ref{fig7}. This rather elaborate $S$-RHP is, after normalizing at $z=0$ and $z\to\infty$ appropriately, amenable to the spectral collocation method
of Olver \cite{Olver12}; we skip the details which can be found in the thesis of the fourth author G.W. \cite[Sect.~5.4]{Wechslberger}
that extends previous work on automatic contour deformation by Bornemann and Wechslberger \cite{IMAJNA,WB14}.
Here, the relative size of the inner and outer circles shaping the contour system shown in the right part of Fig.~\ref{fig7} have to be carefully adjusted to the
parameters $n$ and $m$ to keep the condition number at a reasonable size. The complexity of computing
$f_{n,m}$ for fixed $n$ and $m$ is then basically independent of $m$ and $n$.

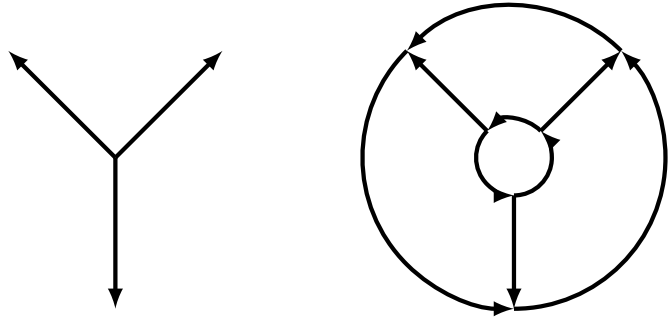
\begin{figure}[tbp]
\centering
\begin{tikzpicture}[rhcontour/.style={->,>=latex,ultra thick}]
\draw[rhcontour] (0.,0) -- (0.,-2.) ;
\draw[rhcontour] (0.,0) -- (1.41421,1.41421) ;
\draw[rhcontour] (0.,0) -- (-1.41421,1.41421)  ;
\end{tikzpicture}
\hfil
\begin{tikzpicture}[rhcontour/.style={->,>=latex,ultra thick}]
\draw[rhcontour] (0.,-0.5) -- (0.,-2.) ;
\draw[rhcontour] (0.353553,0.353553) -- (1.41421,1.41421) ;
\draw[rhcontour] (-0.353553,0.353553) -- (-1.41421,1.41421) ;
\draw[rhcontour](1.41421,1.41421) arc[radius=2., start angle=45., end angle=135.] ;
\draw[rhcontour](-1.41421,1.41421) arc[radius=2., start angle=-225., end angle=-90.] ;
\draw[rhcontour](0.,-2.) arc[radius=2., start angle=-90., end angle=45.] ;
\draw[rhcontour](0.353553,0.353553) arc[radius=0.5, start angle=45., end angle=135.] ;
\draw[rhcontour](-0.353553,0.353553) arc[radius=0.5, start angle=-225., end angle=-90.] ;
\draw[rhcontour](0.,-0.5) arc[radius=0.5, start angle=-90., end angle=45.] ;
\end{tikzpicture}
\caption{Left: Contour of the $S$-RHP \cite[pp. 24--27]{BI15} centered at $z=0$. Right: Modified contour after normalizing the RHP at $z=0$ and $z\to\infty$ (analogously to Fig.~\ref{fig6}); the relative size of the inner and outer circles is chosen depending on $n$ and $m$ and
has a major influence on the condition number of the spectral collocation method. Proper choices steer the condition number
into a regime, which corresponds to a loss of about three to eight digits, see \cite[Sect.~5.4]{Wechslberger}.}\label{fig7}
\end{figure}

In the rest of this work we explore to what extent the analytic transformation from the $X$-RHP to the $S$-RHP is
a necessary preparatory step also numerically, or whether one can use the originally given $X$-RHP as the basis for numerical calculations. To this end, we replace the normalization \eqref{eq:norm_non_standard} by the standard one, that is,
\begin{equation}
X(z) = I + O(z^{-1})\qquad (z\to\infty),
\end{equation}
and introduce a further circle $\Gamma_2$ as shown in the right part of Fig.~\ref{fig6} with the jump condition
\begin{equation}\label{eq:G2}
X_+(\zeta) = G_2(\zeta) X_-(\zeta)\quad (\zeta\in\Gamma_2),\qquad
G_2(\zeta) =
\begin{pmatrix}
\zeta^{\frac{m+n}{2}} & 0 \\*[1mm]
0 & \zeta^{-\frac{m+n}{2}}
\end{pmatrix}.
\end{equation}
\end{subequations}
We define $\Gamma = \Gamma_1 \cup \Gamma_2$ and put $G(\zeta) = G_j(\zeta)$ for $\zeta \in \Gamma_j$ ($j=1,2$).
That way, the Riemann-Hilbert problem is given in the standard form
\begin{equation}\label{eq:XRHP2}
X_+(\zeta) = G(\zeta) X_-(\zeta) \quad (\zeta \in \Gamma),\qquad X(z) = I + O(z^{-1})\quad (z\to \infty).
\end{equation}
Because of $\det G = 1$, the solution $X \in C^\omega(\C\setminus\Gamma, \GL(2))$ is unique, see \cite[p.~104]{Fokas}.

\section{Lower Triangular Jump Matrices and Indices}\label{sect:triang}

We note that the jump matrix $G$ defined in \eqref{eq:G1} and \eqref{eq:G2} is {\em lower triangular}. However, 
even though the non-singular lower triangular matrices form a multiplicative group and the normalization at $z\to\infty$ is also lower
triangular, the solution $X$ turns out to  {\em not} be lower triangular. Arguably the most natural source of RHPs exhibiting this structure are connected to orthogonal polynomials. By renormalizing at $z\to\infty$ the standard RHP for the system of orthogonal polynomials
on the unit circle with complex weight $e^z$, we are led to consider the
following {\em model problem} ($m\in \N$):\footnote{The standard form, see \cite[p.~1124]{BDJ},  of that orthogonal polynomial RHP would be
\[
X_+(\zeta) = \begin{pmatrix}
1 & 0\\
e^\zeta\zeta^{-m} & 1
\end{pmatrix} X_-(\zeta)\quad (|\zeta|=1),\qquad X(z) = 
\begin{pmatrix}
z^m & 0 \\*[1mm]
0 & z^{-m}
\end{pmatrix}(I+O(z^{-1}))\qquad (z\to \infty).
\]
The model problem \eqref{eq:model} is obtained by putting the diagonal scaling at $z\to\infty$ into the jump matrix.}
\begin{equation}\label{eq:model}
Y_+(\zeta) = \begin{pmatrix}
\zeta^m & 0\\
e^\zeta & \zeta^{-m}
\end{pmatrix} Y_-(\zeta)\quad (|\zeta|=1),\qquad Y(z) = I + O(z^{-1}) \quad (z\to \infty).
\end{equation}
Though one could perform a set of transformations to this problem that are standard in
the RHP approach to the asymptotics of orthogonal polynomials on the circle, basically resulting in an analogue
of the $S$-RHP of \cite{BI15}, our point here is to understand the issues of a direct numerical approach
to the $X$-RHP \eqref{eq:XRHP2} in a simple model case.
It is straightforward to check that the {\em unique} solution of \eqref{eq:model} is given explicitly by
\[
Y(z) = \begin{cases}
\begin{pmatrix}
1 \;\;& -z^{-m} e_{m}(-z) \\
0 & 1
\end{pmatrix} &\quad (|z|>1),\\*[6mm]
 \begin{pmatrix}
 z^{m} & -e_{m}(-z) \\
e^z\;\; & z^{-m}(1-e^z e_{m}(-z))
\end{pmatrix}&\quad (|z|<1),
\end{cases}
\]
with 
\begin{equation}\label{eq:exptrunc}
e_k(z) = 1 + z + \frac{z^2}{2!} + \cdots + \frac{z^{k-1}}{(k-1)!} = e^z\frac{\Gamma(k,z)}{\Gamma(k)},
\end{equation}
where $\Gamma(z)$ and $\Gamma(k,z)$ denote the Gamma function and the incomplete Gamma function. 
In particular, we observe that $Y_{12}(0)=-1 \neq 0$.

The nontrivial $12$-component $v$ of a Riemann--Hilbert problem with lower triangular jump matrices,
such as \eqref{eq:XRHP2} or \eqref{eq:model}, can be expressed independently of the other components, it satisfies a homogeneous scalar Riemann--Hilbert problem of its own. Namely, denoting
the $11$-component of $G$ by $g$, we get
\begin{equation}\label{eq:vRHP}
v_+(\zeta) = g(\zeta) v_-(\zeta) \quad (\zeta \in \Gamma),\qquad v(z) = O(z^{-1}) \quad (z\to\infty).
\end{equation}

If the contour is a cycle as in \eqref{eq:XRHP2}, or as in the model problem above, the general theory \cite[§127]{Muskhelishvili} of Riemann--Hilbert problems 
with Hölder continuous boundary regularity
states that the Noether index\footnote{Here, we identify a RHP with an equivalent linear operator equation $Tu = \cdots$, see, e.g., \eqref{eq:RHPSIE} in the next section. We recall that $\lambda = \dim\ker T$ is called the {\em nullity}, $\mu = \dim\coker T$ the
{\em deficiency} and $\kappa = \lambda - \mu$ the {\em Noether index} of a linear operator $T$ with closed range.} $\kappa$ of \eqref{eq:vRHP} 
is given by  the winding number
\[
\kappa = \ind_\Gamma g.
\]
More precisely, the nullity is the sum of the positive partial indices and the deficiency is
the sum of the magnitudes of the negative partial indices, see \cite[Eq.~(127.30)]{Muskhelishvili}. Since there is just one partial index in the scalar case, the nullity of \eqref{eq:vRHP} is $\kappa$ if $\kappa>0$, and the deficiency is $-\kappa$ if $\kappa<0$. 

Thus, in the case of the RHP \eqref{eq:XRHP2}, the nullity of
the scalar sub-RHP for the $12$-component is 
\[
\ind_\Gamma g = \ind_{\Gamma_1} 1 + \ind_{\Gamma_2} \zeta^{(n+m)/2} = \frac{n+m}{2},
\]
 in the case of the model RHP \eqref{eq:model} the corresponding nullity is $m$. 
In both cases, the unique non-zero solution of \eqref{eq:vRHP} that is induced by the solution of the 
defining $2\times 2$ RHP is precisely selected by the compatibility conditions set up by the remaining linear relations of that RHP:
the homogeneous part of these relations must then have Noether index $-\kappa$.

\subsubsection*{Impact on Numerical Methods}

This particular substructure of a Riemann--Hilbert problem with lower triangular jump matrices $G$ is a major challenge for numerical methods. 
If a discretization
of the $2\times 2$ RHP induces a discretization of the scalar subproblem \eqref{eq:vRHP} that results in a homogeneous linear
system with a {\em square} matrix $S_N$ (that is, the same number of equations and unknowns), there are just two (non exclusive) options:
\begin{itemize}
\item $S_N$ is non-singular, which results in a $12$-component $v_N = 0$ that does {\em not} converge;
\item the full system is singular and therefore numerically of not much use (ill-conditioning and convergence issues
will abound). 
\end{itemize} 
Such methods compute fake lower triangular solutions, are ill-conditioned, or both.

To understand this claim, let us denote the $12$-component of the $2\times 2$ discrete solution matrix 
by $v_N$ and the vector of the three other components by $w_N$. By inheriting the subproblem structure such as
\eqref{eq:vRHP} for the $12$-component, the discretization results then in a linear system of the block matrix form
\[
\underbrace{\begin{pmatrix}
S_N & \;0\\*[1mm]
\bullet & T_N 
\end{pmatrix}} _{=A_N}
\begin{pmatrix}
v_N\\*[1mm]
w_N
\end{pmatrix}
=
\begin{pmatrix}
0\\*[1mm]
\bullet
\end{pmatrix}
\]
Because of $\det(A_N) = \det(S_N)\det(T_N)$ a non-singular discretization matrix $A_N$ implies a non-singular $S_N$
and, thus, a non-convergent trivial component $v_N=0$. Such a non-convergent zero $12$-component is what one gets, for example, if one applies the spectral
collocation method of \cite{Olver12} (with square contours replacing the circles) to the Riemann--Hilbert problems \eqref{eq:XRHP2} or to the model problem \eqref{eq:vRHP}. As a hint of failure, the resulting discrete system is ill-conditioned; details which can be found in the thesis of the fourth author G.W. \cite[Sect.~5.4]{Wechslberger}. 

The deeper structural reason for this problem can be seen in the fact that the Noether index of finite-dimensional square matrices is always zero,
whereas the index of the infinite-dimensional subproblem \eqref{eq:vRHP} is strictly positive.

We suggest two approaches to deal with this problem: first, an infinite-dimensional discretization using sequence spaces, that is, without truncation,
and  using infinite-dimensional numerical linear algebra, and second, using underdetermined discretizations with
rectangular linear systems that are complemented by a set of explicit compatibility conditions.

\section{RHPs as Integral Equations with Singular Kernels}

In this section, we recall a way to express the RHP \eqref{eq:XRHP2}, with standard normalization at infinity, as a particular system of singular integral equations, cf. \cite{RHPFact,DeiftZhou,Olver12}.
We introduce the Cauchy transform
\[
C f(z) = \frac{1}{2\pi i}\int_\Gamma \frac{f(\zeta)}{\zeta-z}\,d\zeta \qquad (z\not\in\Gamma)
\]
and their directional limits $C_\pm$ when approaching the left or right of the oriented contour $\Gamma$, defined by
\[
C_\pm f(\eta) = \lim_{z \to \eta_\pm} \frac{1}{2\pi i}\int_\Gamma \frac{f(\zeta)}{\zeta-z}\,d\zeta \qquad (\eta \in \Gamma).
\]
Note that $C_\pm$ can be extended as bounded linear operators mapping $L^2(\Gamma)$ (or spaces of Hölder continuous functions) into itself, 
and $C$ (suitably extended) maps such functions into functions that are holomorphic on $\C\setminus\Gamma$, see \cite[p.~100]{Fokas}. By using the decomposition $C_+ - C_- = \1$, the ansatz (by letting $C$ act component-wise on the matrix-valued function $u$)
\begin{subequations}\label{eq:RHPSIE}
\begin{equation}\label{eq:cauchyIntegral}
X(z) = I + C u(z),\qquad u \in L^2(\Gamma,\C^{2\times 2}),
\end{equation}
establishes the equivalence of a RHP of the form \eqref{eq:XRHP2} and the system of singular integral equations
\begin{equation}\label{eq:SIE}
(\1 - (G-I) C_-) u = G-I
\end{equation}
\end{subequations}
As the following theorem shows, singular integral operators of the form
\[
T_G = \1 - (G-I) C_- : L^2(\Gamma,\C^{2\times 2}) \to L^2(\Gamma,\C^{2\times 2})
\]
 can be preconditioned by operators of exactly the same form.

\begin{theorem}\label{thm:precond} 
Let $\Gamma$ be a smooth, bounded, and non-self intersecting\footnote{Points of self intersection are 
allowed if certain cyclic conditions are satisfied \cite{FokasZhou}: at such a point
the product of the corresponding parts of the jump matrix should be the
identity matrix. These conditions guarantee smoothness in the sense of \cite{Zhou}, where  
the analog of Theorem~\ref{thm:precond} is proved for the general smooth Riemann--Hilbert data.} contour system and $G:\Gamma \to \GL(2)$ a system of jump matrices
which continues analytically to a vicinity of $\Gamma$. Then,  $T_{G^{-1}}$ is a Fredholm regulator of $T_G$, that is, $T_{G^{-1}} T_G = \1 + K$
with a compact operator $K : L^2(\Gamma,\C^{2\times 2}) \to L^2(\Gamma,\C^{2\times 2})$  that can be represented as a regular integral operator.
\end{theorem}

\begin{proof} The Sokhotski--Plemelj formula \cite[Eq. (17.2)]{Muskhelishvili} gives that $2C_- = -\1 + H$, where $H$ denotes a variant of the Hilbert transform (normalized as in \cite{Muskhelishvili}),
\[
H f(\zeta) = \frac{1}{\pi i} \int_\Gamma \frac{f(\eta)}{\eta - \zeta}\,d\eta\qquad (\zeta \in \Gamma),
\]
with the integral understood in the sense of principle values. This way, we have
\begin{align*}
T_G &= A_1 \1 + B_1 H,\qquad A_1 = \frac{1}{2} (I+G),\quad B_1 = \frac{1}{2} (I-G),\\*[2mm]
T_{G^{-1}} &= A_2 \1 + B_2 H,\qquad A_2= \frac{1}{2} (I+G^{-1}),\quad B_2 = \frac{1}{2} (I-G^{-1}).
\end{align*}
By a product formula of Muskhelishvili  \cite[Eq. (130.15)]{Muskhelishvili}, which directly follows from the Poincaré--Betrand formula \cite[Eq. (23.8)]{Muskhelishvili}, one has
\[
T_{G^{-1}} T_G = A \1 + B H + K,
\]
where $K$ represents a {\em regular} integral operator and the coefficient matrices $A$ and $B$ are given by the expressions
\[
A = A_2 A_1 + B_2 B_1,\qquad B = A_2 B_1 + B_2 A_1.
\]
Here, we thus obtain $A=I$ and $B=0$, which finally proves the assertion.\qed
\end{proof}

This theorem implies that the operator $T_{G}$ is {\em Fredholm}, that is, its nullity and
deficiency are {\em finite}.  In fact, since in our examples
$\det G \equiv 1$, we have that the Noether index of $T_G$ is zero. The possibility  to use the  Fredholm theory is extremely important in studying RHPs: it allows one to use, when
proving the solvability of Riemann-Hilbert problems, the ``vanishing lemma'' \cite{Zhou}, see also \cite[Chap.~5]{Fokas}.
For the use of Fredholm regulators in iterative methods applied to solving singular integral equations, see \cite{MR3348205}.

\section{A Well-Conditioned Spectral Method for Closed Contours}\label{sect:spectral} 

We follow the ideas of Olver and Townsend \cite{Spectral} on spectral methods for differential
equations, recently extended by Olver and Slevinsky \cite{SIEpaper} to singular integral equations.
First, the solution $u$ and the data $G-I$ of the singular integral equation \eqref{eq:SIE} 
are expanded\footnote{It is actually implemented this way in {\tt SingularIntegralEquations.jl}, a {\sc Julia} software package  described in \cite{SIEpaper}.} in the Laurent bases of the circles that built up the cycle $\Gamma$.
Next, the resulting  linear 
system is solved using the framework of {\em infinite-dimensional} linear
algebra \cite{NLA,PracticalFramework}, built out of the adaptive QR factorization introduced in \cite{Spectral}. 

To be specific, we describe
the details for the model RHP \eqref{eq:model}, where the cycle $\Gamma$ is just the unit circle. Here, we have
the expansions
\[
u(\zeta) = \sum_{k=-\infty}^\infty U_k \zeta^k,\qquad G(\zeta)-I = \sum_{k=-\infty}^\infty A_k \zeta^k\qquad (\zeta\in \Gamma),
\]
both rapidly decaying with $2\times 2$ coefficient matrices $U_k$ and $A_k$.  In the Laurent basis, the
operator $C_-$ acts diagonally in the simple form
\[
C_- \zeta^k 
=\begin{cases}
0 & \quad k \geq 0,\\*[2mm]
- \zeta^k & \quad k < 0.
\end{cases}
\]
which gives
\[
C_- u(\zeta) = - \sum_{k=1}^\infty U_{-k} \zeta^{-k}\qquad (\zeta\in\Gamma).
\]
Note that $-C_-$ acts as a projection to the subspace spanned by the basis elements with negative index. 
This way, the system \eqref{eq:SIE} of singular integral equations is transformed to\footnote{We use the
Iverson bracket of a condition: $[{\mathcal P}] = 1$ if the predicate ${\mathcal P}$ is true,  $[{\mathcal P}] = 0$ otherwise.}
\begin{equation}\label{eq:SIEdiscr}
U_k + \sum_{j=-\infty}^{\infty}  [k-j<0]\, A_{j} U_{k-j}= A_k \qquad (k\in\Z).
\end{equation}

Up to a given accuracy, we may assume that the data is given as a finite sum,
\[
G(\zeta)-I \approx \sum_{k=-n_1}^{n_1} A_k \zeta^k,
\]
likewise for $G^{-1}(\zeta) - I$ with a truncation at $n_2$. Thus, writing the discrete system \eqref{eq:SIEdiscr}
in matrix-vector form, the corresponding double-infinite matrix has a bandwidth of order $O(n_1)$. Preconditioning
this system, following Theorem~\ref{thm:precond}, by the multiplication with 
the double-infinite matrix belonging to $G^{-1}$ instead of $G$, results in a double-infinite matrix that has a
 bandwidth of order $O(n_1+n_2)$. Since the right hand side of \eqref{eq:SIEdiscr} is truncated at indices of magnitude $O(n_1)$,
application of  the adaptive QR factorization \cite{Spectral}, after re-ordering the double-infinite coefficients as $U_0,U_{-1},U_1,\ldots$ in order to be singly infinite, will result in an algorithm that has a complexity
of order $O((n_1+n_2)^2 n_3)$, where $n_3$ is the number of coefficients needed to resolve $u$, dictated by a specified tolerance.   

\begin{remark} The extension to systems $\Gamma$ of closed contours built from several circles is straightforward.
The jump data and the solution, restricted to a circle centered at $a$ are expanded in the Laurent basis $(z-a)^k$, $k\in\Z$.
When instead evaluated at a circle centered at $b$, a change of basis is straightforwardly computed using
\[
(z-a)^j = \sum_{k=0}^\infty \binom{j}{k} (b-a)^{j-k}(z-b)^k\qquad (j\in\Z),
\]
valid for $|z-b| < |b-a|$. Because of a geometric decay, one can truncate those series at $k = O(1)$ as long as
$|z-b| < \theta |b-a|$ with $0<\theta < 1$ small enough.  The adaptive QR factorization can then be applied by interlacing the Laurent coefficients on each circle to obtain a singly infinite unknown vector of coefficients.
\end{remark}

\subsubsection*{Numerical Example 1: Model problem}

Because of the entries $\zeta^m$ and $\zeta^{-m}$ in the jump matrix of the model problem \eqref{eq:model}, 
we have that $n_1, n_2, n_3 = O(m)$ in order to resolve the data and the solution; hence the computational complexity of the method scales as $O(m^3)$. Using the {\sc Julia} software package {\tt SingularIntegralEquations.jl}\footnote{{\tt https://github.com/ApproxFun/SingularIntegralEquations.jl}, cf. \cite{SIEpaper}.} (v0.0.1) the problem is numerically solved by
the following short code showing that the user has to do little more than just providing the data and
entering the singular integral equation \eqref{eq:SIE} as a mathematical expression: 

\begin{lstlisting}[numbers=left,basicstyle=\ttfamily\footnotesize,mathescape=true]
using ApproxFun, SingularIntegralEquations

m = 100
$\Gamma$ = Circle(0.0,1.0)
G = Fun(z -> [z^m 0; exp(z) 1/z^m],$\Gamma$)
C = Cauchy(-1)
@time u = (I-(G-I)*C)\(G-I)
Y = z -> I+cauchy(u,z)
err = norm(Y(0)-[0 -1; 1 (-1)^m*exp(-lfact(m))],2)
\end{lstlisting}
The run time\footnote{Using a MacBook Pro with a 3.0 GHz Intel Core i7-4578U processor and 16 GB of RAM.} 
is 2.8 seconds, the
error of $Y(0)$ is $4.22\cdot 10^{-15}$ (spectral norm), which corresponds to a loss of one digit in absolute error.

\subsubsection*{Numerical Example 2: Riemann--Hilbert Problem for the Discrete $Z^{2/3}$}

Now, we apply the method to the Riemann--Hilbert problem \eqref{eq:XRHP2} encoding the discrete $Z^a$ map. 
Here, because of the exponents $-m$, $-n$ and $\pm (n+m)/2$ in \eqref{eq:XRHP}, we have $n_1, n_2, n_3 = O(n+m)$ in order to resolve the data and the solution, see Fig.~\ref{fig8}; hence the computational complexity scales as $O((n+m)^3)$.
Note that this is far from optimal, using the stabilized recursion of Section~\ref{sect:stable} to
compute a table including $Z_{n,m}^a$ would give a complexity of order $O((n+m)^2)$. Once more, however,
the code requires little more than typing the mathematical equations of the RHP.

\begin{lstlisting}[numbers=left,basicstyle=\ttfamily\footnotesize,mathescape=true]
using ApproxFun, SingularIntegralEquations

a = 2/3
n = 6; m = 8; # n+m must be even
pow = z -> exp(1im*a*pi/4)*exp(-a/2*log(z/1im))
$\Gamma$ = Circle(-1.0,0.3) $\cup$ Circle(+1.0,0.3) $\cup$ Circle(0.0,3.0) 
G = Fun(z -> in(z,$\Gamma$[3])?[z^((m+n)/2) 0; 0 1/z^((n+m)/2)]:[1 0; pow(z)/(z-1)^m/(z+1)^n 1],$\Gamma$)
C = Cauchy(-1)
@time u = (I-(G-I)*C)\(G-I)
X = z -> I+cauchy(u,z);
X0 = X(0)
Za0 = (-1)^(m+1)*X0[2,1]/X0[1,1]
Za1 = 3.610326860525178 + 2.568086087959661im # exact solution from the recursion as in Section 1.2 using bigfloats
err = abs(Za0 - Za1) 
\end{lstlisting}
The run time is 2.7 seconds, the absolute
error of $Z^{2/3}_{6,8}$ is $3.38\cdot 10^{-8}$, which corresponds to a loss of about 7 digits. This loss
of accuracy can be explained by comparing the magnitude of the $21$-component of $u$ as shown in Fig.~\ref{fig8}, along the
two black circles of Fig.~\ref{fig6}, with that of the corresponding component of the
solution matrix at $z=0$, namely, 
\[
X(0) \approx 
\begin{pmatrix}
-3.38121       &        -12.2073+8.68324i \\
           12.2073+8.68324i &  66.0758
\end{pmatrix}.
\]
We observe that during the evaluation of the Cauchy transform \eqref{eq:cauchyIntegral}, which maps $u\mapsto X(0)$
by means of an integral, at least 5 digits must have been lost by cancellation---a loss, which structurally cannot be avoided for oscillatory
integrands with {\em large} amplitudes. (Note that this is not an issue of frequency: just one oscillation with a large
amplitude suffices to get such a severe cancellation.) 

\begin{figure}[tbp]
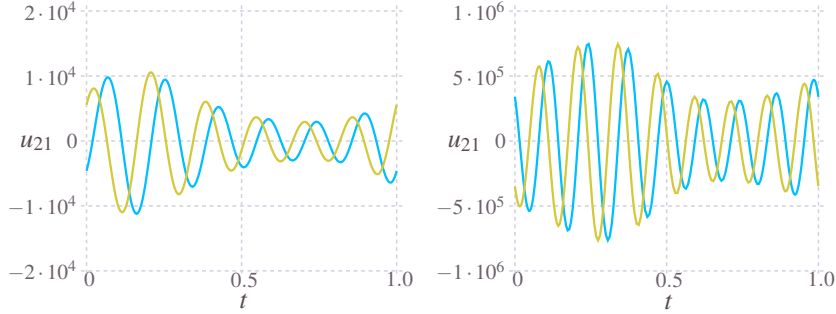

\centering
\begin{minipage}{0.45\textwidth}
\input{Images/circle_m1}
\end{minipage}\hfil
\begin{minipage}{0.45\textwidth}
\input{Images/circle_p1}
\end{minipage}
  \caption{Left: $u_{21}(\zeta)$ on the circle $\zeta = -1 + 0.3 e^{2\pi i t}$; right: $u_{21}(\zeta)$ on $\zeta = +1 + 0.3 e^{2\pi i t}$. The real parts are shown in blue, the imaginary part in yellow. Note that there are $m=6$ oscillations
  on the left and $n=8$ oscillations on the right; the maximum amplitude is about $1.1\cdot 10^4$ on the left and $7.5\cdot 10^5$ on the right.}\label{fig8}
 \end{figure}
 
Since the amplitudes of $u_{21}$ grow exponentially with $n$ and $m$, the algorithm for computing $Z^a_{n,m}$ based on the numerical evaluation of
\eqref{eq:RHPSIE} applied to the RHP \eqref{eq:XRHP2} is {\em numerically unstable}. Even though the initial step,
the spectral method in coefficient space applied to \eqref{eq:SIE} is perfectly stable, stability is destructed by
the bad conditioning of the post-processing step, that is, the evaluation of the integral in \eqref{eq:cauchyIntegral}. 
We refer to \cite{stability} for an analysis that algorithms with a badly conditioned post processing of
intermediate solutions are generally prone to numerical instability.
   
\section{RHPs as Integral Equations with Nonsingular Kernels}

By reversing the orientation of the two small circles in the RHP \eqref{eq:XRHP2}, and by simultaneously replacing the
jump matrix $G_1$ by $\tilde G_1 = G_1^{-1}$, the RHP is transformed to an equivalent one with a contour system $\Sigma$
that satisfies the following properties, see Fig.~\ref{fig9}: it is a union of non-self intersecting smooth curves, that 
bound a domain $\Omega_+$ to its left. By $\Omega_-$ we will denote the (generally not connected) region which is the
complement of $\Omega_+ \cup \Gamma$. Note that the model problem \eqref{eq:model} falls into that class of contours without any
further transformation.

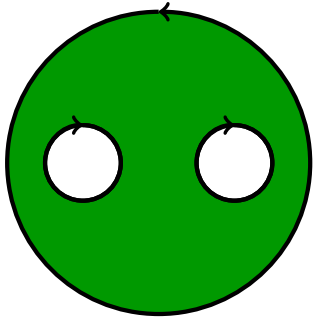
\begin{figure}[tbp]
\centering
	\begin{tikzpicture}[scale=2, transform shape,rhcontour/.style={->,>=to,ultra thick}]
   		\draw[rhcontour,fill=darkgreen,even odd rule] (-0.5,0) arc[radius=0.25, start angle=90, end angle=450] 
		     (0.5,0) arc[radius=0.25, start angle=90, end angle=450] 
		     (0,0.75) arc[radius=1, start angle=90, end angle=450]; 
   		\draw[rhcontour] (0.5,0) arc[radius=0.25, start angle=90, end angle=-270];
   		\draw[rhcontour]	(-0.5,0) arc[radius=0.25, start angle=90, end angle=-270];
	\end{tikzpicture}
\caption{A modified, but equivalent, contour system $\Sigma$ for the for the RHP \eqref{eq:XRHP2} obtained
by reversing the orientation and by simultaneously replacing the jump matrix $G_1$ by $G_1^{-1}$. Now, there is a bounded domain $\Omega_+$ (marked in green) to the left of $\Sigma$, cf. the original system shown in Fig.~\ref{fig6}.}\label{fig9}
\end{figure}

We drop the tilde from the jump matrices and consider RHPs  of the form
\begin{equation}\label{eq:PhiRHP}
\Phi_+(\zeta) = G(\zeta) \Phi_-(\zeta) \quad (\zeta \in \Sigma),\qquad \Phi(z) = I + O(z^{-1})\quad (z\to \infty),
\end{equation}
on such contours systems $\Sigma$.
It will either represent the aforementioned transformation of \eqref{eq:XRHP2} or the model problem \eqref{eq:model}.
In particular, $G$ is lower triangular and can be analytically continued to a vicinity of $\Sigma$.

The classical theory developed by Plemelj (see \cite[§126]{Muskhelishvili}) for such problems 
teaches the following: the directed boundary values $\Phi_-$ of the unique analytic solution 
$\Phi: \C\setminus\Sigma \to \GL(2)$ of the RHP \eqref{eq:PhiRHP} satisfy a Fredholm integral equation \cite[Eq.~(126.5)]{Muskhelishvili} of the second kind on $\Sigma$, namely 
\begin{equation}\label{eq:fred}
\Phi_-(\zeta) - \frac{1}{2\pi i}\int_\Sigma \frac{G^{-1}(\zeta) G(\eta) - I}{\eta - \zeta} \Phi_-(\eta)\,d\eta = I\qquad (\zeta\in\Sigma),
\end{equation}
understood here as an integral equation in $L^2(\Sigma,\C^{2\times 2})$.
The matrix kernel of this equation, that is,
\[
K(\zeta,\eta) = \frac{G^{-1}(\zeta) G(\eta) - I}{\eta - \zeta} = G^{-1}(\zeta)\frac{ G(\eta) - G(\zeta)}{\eta - \zeta},
\]
is smooth on $\Sigma\times \Sigma$, since it extends as an analytic function and since the singularity at $\zeta=\eta$ is removable. Integral equations of the form \eqref{eq:fred} with a smooth kernel are, in principle, amenable to fast quadrature
based methods, see the next section. 

We note that, given the boundary values $\Phi_-(\zeta)$ for $\zeta \in \Sigma$, the solution of the RHP \eqref{eq:PhiRHP}
can be reconstructed by
\begin{equation}
\Phi(z) = \begin{cases}
I - \dfrac{1}{2\pi i} \displaystyle\int_\Sigma \dfrac{\Phi_-(\zeta)}{\zeta-z}\,d\zeta &\qquad z \in \Omega_-,\\*[5mm]
\dfrac{1}{2\pi i} \displaystyle\int_\Sigma \dfrac{G(\zeta)\Phi_-(\zeta)}{\zeta-z}\,d\zeta&\qquad z \in \Omega_+.
\end{cases}
\end{equation}

In general, however, the Fredholm equation \eqref{eq:fred} is not equivalent to the RHP, see \cite[p.~387]{Muskhelishvili}: the Fredholm equation has but a kernel of the same dimension as the
 kernel of the {\em associated} homogeneous RHP, defined as
\begin{equation}\label{eq:PsiRHP}
\Psi_+(\zeta) = G^{-1}(\zeta)  \Psi_-(\zeta)\quad (\zeta \in \Sigma),\qquad \Psi(z) = O(z^{-1})\quad (z\to \infty).
\end{equation}
As we will show now, the kernel of the associated RHP is {\em nontrivial} in the examples studied in this work.

First, we observe, by the lower triangular form of $G$, that the $11$- and the $12$-components of $\Psi$ both satisfy a
scalar RHP of the form \eqref{eq:vRHP} with a jump function $g$ that has a winding number which is
\[
\ind_\Sigma g = - \frac{n+m}{2} 
\]
for the discrete map $Z^a$, and which is  $\ind_\Sigma g = -m$ for the model problem \eqref{eq:model}. Note that this
winding number has the sign opposite to the results of Section~\ref{sect:triang}
since the underlying $2\times 2$ RHP is based on $G^{-1}$ instead of $G$. Hence, the nullity of the scalar RHPs for
the  $11$- and the $12$-components of $\Psi$ is zero and the deficiency is $(n+m)/2$ ($m$ in case of the model problem).
As a consequence, the $11$- and the $12$-components of $\Psi$ must both be identically zero.

Next, since we now know that $\Psi$ has a zero first row, also the $21$- and $22$-components of $\Psi$ satisfy a
scalar RHP of the form \eqref{eq:vRHP} each, but with a jump function $g$ that has the positive winding number 
\[
\ind_\Sigma g =  \frac{n+m}{2} 
\]
for the discrete map $Z^a$, and $\ind_\Sigma g = m$ for the model problem \eqref{eq:model}, just as discussed in Section~\ref{sect:triang}. Hence, the deficiency of the scalar RHPs for
the  $21$- and the $22$-components of $\Psi$ is zero and the nullity is $(n+m)/2$ ($m$ in case of the model problem).
Since both components are linearly independent from of each other, we have thus proven the following lemma.

\begin{lemma}\label{lem:1} The nullity of the associated homogeneous RHP \eqref{eq:PsiRHP}, and hence, that of the Fredholm integral equation 
\eqref{eq:fred} is $n+m$ in the case of the discrete map $Z^a$ and $2m$ in the case of the model problem \eqref{eq:model}. 
\end{lemma}

\begin{example} For the model RHP \eqref{eq:model} the smooth kernel of the Fredholm integral equation \eqref{eq:fred} can be constructed explicitly. Here we have
\[
K(\zeta,\eta) = \frac{G^{-1}(\zeta) G(\eta) - I}{\eta - \zeta} = 
\begin{pmatrix}
\frac{(\eta/\zeta)^m-1}{\eta-\zeta} & \;0 \\*[2mm]
\frac{e^\eta \zeta^m - e^\zeta \eta^m}{\eta-\zeta} & \;\frac{(\zeta/\eta)^m-1}{\eta-\zeta}
\end{pmatrix}.
\]
A column of a matrix belonging to the kernel of \eqref{eq:fred} satisfies the equation
\begin{equation}\label{eq:kernel}
\begin{pmatrix}
u_-(\zeta) \\
w_-(\zeta) 
\end{pmatrix} = \frac{1}{2\pi i} \int_\Sigma K(\zeta,\eta)  
\begin{pmatrix}
u_-(\eta) \\
w_-(\eta) 
\end{pmatrix}\,d\eta\qquad (\zeta\in\Sigma),
\end{equation}
where $\Sigma$ is the positively oriented unit circle. We will construct solutions that extend analytically as $u_-(z)$ and $w_-(z)$ for $z\neq 0$, such that
\[
 \begin{pmatrix}
u_-(z) \\
w_-(z) 
\end{pmatrix} = \res_{\eta = 0 } K(z,\eta)  
\begin{pmatrix}
u_-(\eta) \\
w_-(\eta) 
\end{pmatrix}\qquad (z \neq 0).
\]
\end{example}
By recalling the notation introduced in \eqref{eq:exptrunc}) we observe, for $k=0,\ldots,m-1$, that
\begin{align*}
\res_{\eta=0}K_{11}(z,\eta) \eta^{k-m} &= z^{k-m}, \\*[2mm]
\res_{\eta=0}K_{21}(z,\eta) \eta^{k-m} &= -z^k e_{m-k}(z),\\*[2mm]
\res_{\eta=0}K_{22}(z,\eta) \eta^k &= -z^k.
\end{align*}
Using the coefficients $a_{jk}^{(m)}$ that induce a change of polynomial basis by
\[
z^k = \sum_{k=0}^{m-1} a_{kj}^{(m)} z^j e_{m-j}(z)\qquad (k=0,\ldots,m-1),
\]
we define the polynomials
\[
p_k^{(m)}(z) = \sum_{j=0}^{m-1} a_{kj}^{(m)} z^j \qquad (k=0,\ldots,m-1),
\]
each of which has degree at most $m-1$.
Then, the $m$ linear independent vectors
\begin{equation}\label{eq:kercol}
\begin{pmatrix}
u_-(\zeta) \\
w_-(\zeta) 
\end{pmatrix} = 
\begin{pmatrix}
-2 \zeta^{-m} p_k^{(m)}(\zeta) \\*[2mm]
 \zeta^k
\end{pmatrix}\qquad (k=0,\ldots,m-1)
\end{equation}
are solutions of \eqref{eq:kernel} each. Thus, since its dimension is $2m$ by Lemma~\ref{lem:1}, the kernel of the integral
equation \eqref{eq:fred} is spanned by the $2\times 2$ matrices whose columns are linear
combinations of these vectors. 
\qed

The unique solution $\Phi_-$ of the RHP \eqref{eq:PhiRHP} can be picked among the solutions of \eqref{eq:fred} by imposing 
additional linear conditions, namely $n+m$ independent such 
conditions in the case of the discrete map $Z^a$ and $2m$ in the case of the model problem. 
Specifically, for the model problem \eqref{eq:model}, we obtain such conditions as follows. First, since $\Phi_-(z)$ continues analytically to
$|z|>1$ and since $\Phi_-(z) = I + O(z^{-1})$ as $z\to \infty$, we get by Cauchy's formula for the Laurent coefficients at $z=\infty$ that
\[
\frac{1}{2\pi i} \int_\Sigma \Phi_-(\zeta)\frac{d\zeta}{\zeta^k} = [k=1]\cdot I\qquad (k=1,2,\ldots).
\]
Second, by restricting this relation to the {\em second row} of the matrix $\Phi_-$ for $k=1,\ldots,m$, we get
the conditions
\begin{equation}\label{eq:conditions}
\frac{1}{2\pi i} \int_\Sigma 
\begin{pmatrix}
0 &\;\; 1 
\end{pmatrix} \cdot \Phi_{-}(\zeta)\,\frac{d\zeta}{\zeta^k} = 
\begin{pmatrix}
0\;\; & [k=1] 
\end{pmatrix}\qquad (k=1,\ldots,m).
\end{equation}
In fact, these conditions
force all the components of the columns \eqref{eq:kercol} that would span an offset from the kernel of \eqref{eq:fred} to be zero.

For the $Z^a$-RHP, similar arguments prove that the kernel of \eqref{eq:PhiRHP} is spanned by matrices
whose second row extends to polynomials of degree smaller than $(n+m)/2$  to the outside of the outer circle in
in Fig.~\ref{fig9}.  Thus, the same form of conditions as in \eqref{eq:conditions} can be applied for
picking the proper solution $\Phi_-(\zeta)$, except that one would have to replace
$\Sigma$ by that outer circle and the upper index $m$ by $(n+m)/2$.

\section{A Modified Nyström Method}\label{sect:ny}

Fredholm integral equations of the second kind with smooth kernels defined on a system of circular contours
are best discretized by the classical Nyström method \cite[§12.2]{Kress}. Here, one uses the composite trapezoidal rule as the underlying quadrature formula, that is,
\[
\frac{1}{2\pi i}\int_{\partial B_r(z_0)} f(z)\,dz \approx \frac{r}{N} \sum_{j=0}^{N-1} f\left(z_0 + r e^{2\pi i j/N}\right) e^{2\pi i j/N}.
\]
For integrands that extend analytically to a vicinity of the contour, this quadrature formula is spectrally accurate,
see, e.g., \cite[§2]{FoCM11} or \cite[§2]{SIREV}.

Since the Fredholm integral equation \eqref{eq:fred} has a positive nullity, applying the Nyström method to it will yield, for $N$ large enough, a numerically singular
linear system. However, the theory of the last section suggests a simple modification of the Nyström method: we use the conditions \eqref{eq:conditions} (after approximating them by the same quadrature formula as for the Nyström method) as
additional equations and solve the resulting {\em overdetermined} linear system by the least squares method. 

\subsubsection*{Numerical Example 1: Model problem}

We apply the modified Nyström method to the Fredholm integral equation representing the model problem \eqref{eq:model}. By the sampling condition, see, e.g., \cite[§2]{FoCM11}, the number $N$ of quadrature points will scale as $N = O(m)$, hence the computational complexity scales as $O(m^3)$.
To check the accuracy we compare with
\[
Y(0) = \dfrac{1}{2\pi i} \displaystyle\int_\Gamma G(\zeta)\Phi_-(\zeta)\dfrac{d\zeta}{\zeta}, \qquad
I = \dfrac{1}{2\pi i} \displaystyle\int_\Gamma\Phi_-(\zeta)\dfrac{d\zeta}{\zeta},
\]
evaluated by the same quadrature formula as for the Nyström method. For the particular parameters $m=100$ and $N=140$ we get,
within a run-time of 0.49 seconds for a straightforward Matlab implementation, 
a maximum error of these two quantities, measured in $2$-norm, of $1.33\cdot 10^{-14}$. The condition number
of the least squares matrix grows just moderately with $m$: it is about $23$ for $m=1$ and about $650$ for $m=1000$.

\begin{figure}[tbp]
\centering
\setlength\figurewidth{0.575\textwidth}
% This file was created by matlab2tikz v0.5.0 running on MATLAB 8.6.
%Copyright (c) 2008--2014, Nico Schlömer <nico.schloemer@gmail.com>
%All rights reserved.
%Minimal pgfplots version: 1.3
%
\begin{tikzpicture}

\begin{axis}[%
width=0.950920245398773\figurewidth,
height=0.75\figurewidth,
at={(0\figurewidth,0\figurewidth)},
scale only axis,
separate axis lines,
every outer x axis line/.append style={white!15!black},
every x tick label/.append style={font=\color{white!15!black}},
xmin=10,
xmax=50,
xlabel={$N_0$},
xmajorgrids,
every outer y axis line/.append style={white!15!black},
every y tick label/.append style={font=\color{white!15!black}},
ymode=log,
ymin=1e-12,
ymax=1,
yminorticks=true,
ylabel={error},
ymajorgrids,
yminorgrids
]
\addplot [color=blue,line width=0.8pt,mark size=1.7pt,only marks,mark=*,mark options={solid},forget plot]
  table[row sep=crcr]{%
10	0.2744049409775\\
11	0.255969261004664\\
12	0.166633454658061\\
13	0.184634631201704\\
14	0.101444737617087\\
15	0.0409410780759812\\
16	0.0868474685789512\\
17	0.0945038685668819\\
18	0.112287570143139\\
19	0.0938726910280181\\
20	0.130128360132468\\
21	0.0509729279280376\\
22	0.0396670440171682\\
23	0.0141935615823551\\
24	0.00705249341268218\\
25	0.00296026214605477\\
26	0.00130229643602803\\
27	0.000558548988336194\\
28	0.000236144240810048\\
29	0.000100188546897249\\
30	4.15008650580273e-05\\
31	1.73356831025208e-05\\
32	7.07354400476655e-06\\
33	2.91237078168409e-06\\
34	1.17267535610634e-06\\
35	4.76154717586308e-07\\
36	1.89538695323033e-07\\
37	7.61508312791499e-08\\
38	3.02081048848887e-08\\
39	1.16311497866717e-08\\
40	4.33050637489334e-09\\
41	1.43261804054282e-09\\
42	1.13229983364233e-09\\
43	1.30255741508167e-10\\
44	4.97007402825224e-10\\
45	2.02035484690409e-10\\
46	4.95972142825115e-10\\
47	4.19822377996717e-10\\
48	8.2980784153681e-11\\
49	2.80294016864763e-10\\
50	6.96649832277218e-10\\
};
\end{axis}
\end{tikzpicture}%
\caption{Absolute error of the approximation of $Z^{2/3}_{6,8}$ by the modified Nyström method vs. the number
of quadrature points $N_0$ on each of the three circles in Fig.~\ref{fig9} (the radii  are $1/2$ for the inner circles, $3$ for the outer one); the total number of quadrature points is then $N=3\times N_0$. One observes, after a threshold caused by a sampling condition, exponential (i.e., spectral) convergence that
saturates at a level of numerical noise at an error of about $10^{-9}$. Run time of a Mathematica implementation with $N_0=42$ is about 0.15 seconds.}\label{fig10}
\end{figure}
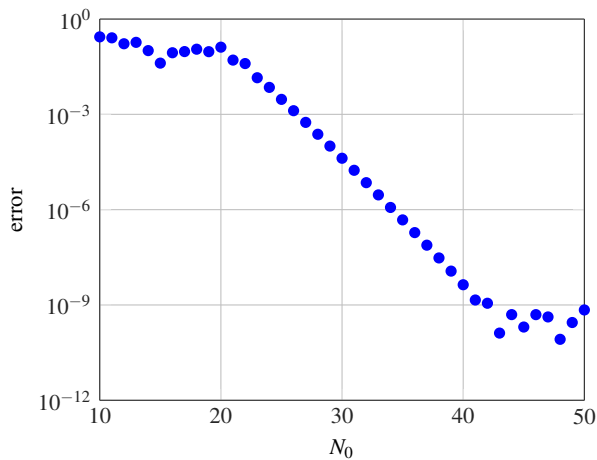

\subsubsection*{Numerical Example 2: Discrete $Z^{2/3}$}

Now, we apply the modified Nyström method to the Fredholm integral equation representing the RHP \eqref{eq:XRHP2} subject
to a transformation to the form \eqref{eq:PhiRHP}. Here, the sampling condition requires $N = O(n+m)$, hence the
computational complexity scales as $O((n+m)^3)$. For $Z_{6,8}^{2/3}$, the modified Nyström method yields the convergence plot shown in Fig.~\ref{fig10}: it exhibits exponential (i.e., spectral) convergence until
a noise level of about $10^{-9}$ is reached, which corresponds to a loss of about 6 digits. The reason for this loss is that this method for approximating the discrete $Z^a$ 
suffers the same issue with a bad conditioning of the post-processing step, that is, of 
\[
\Phi_-(\cdot) \mapsto X(0) = \dfrac{1}{2\pi i} \displaystyle\int_\Gamma G(\zeta)\Phi_-(\zeta)\dfrac{d\zeta}{\zeta},
\]
as the spectral method for the singular integral equation discussed in Section~\ref{sect:spectral}. Here, the amplitude
of the real and imaginary part of $\Phi_-(\zeta)$ along the two inner circles is of the order $10^4$ which causes 
a cancellation of at least 4 significant digits.

\section{Conclusion}

To summarize, there are two fundamental options for the stable numerical evaluation of the discrete map $Z^a_{n,m}$. 

\begin{itemize}
\item Computing all the values of the array $1\leq n,m\leq N$ at once by, first, computing the diagonal using a boundary value
solve for the discrete Painlevé II equation~\eqref{eq:xn} and, then, by recursing from the diagonal to the boundary using
the discrete differential equation \eqref{eq2}. This approach  has optimal complexity $O(N^2)$.
\item Computing just a single value for a given index pair $(n,m)$ by using the RHP~\eqref{eq:XRHP2} and one of the methods
discussed in Sections~\ref{sect:spectral} or \ref{sect:ny}. Since both methods suffer from an instability caused by
a post-processing quadrature for larger values of $n$ and $m$, one would rather mix this approach with the asymptotics \eqref{eq:asympt}.
For instance, using the numerical schemes for $n,m \leq 10$, and the asymptotics otherwise, gives a uniform precision
of about  5 digits for $a=2/3$. Higher accuracy would require the calculation of the next order terms of the
asymptotics as in Section~\ref{sect:stable}. Obviously, this mixed numerical-asymptotic method has optimal complexity $O(1)$.
\end{itemize}

\begin{acknowledgement}
The research of F.B., A.I., and G.W. was supported by the DFG-Collaborative Research Center, TRR 109, ``Discretization in Geometry and Dynamics.''
\end{acknowledgement}

\bibliographystyle{spmpsci}
\bibliography{Zalpha}

\bigskip

{\footnotesize
\noindent 
(Folkmar Bornemann) {\sc Zentrum Mathematik -- M3, Technische Universität München, Germany}\\*[0.5mm]
{\em E-mail address}: {\tt bornemann@tum.de}

\medskip

\noindent 
(Alexander Its) {\sc Department of Mathematical Sciences, Indiana University--Purdue University, Indianapolis, USA}\\*[0.5mm]
{\em E-mail address}: {\tt itsa@math.iupui.edu}

\medskip

\noindent 
(Sheehan Olver) {\sc School of Mathematics and Statistics, The University of Sydney, Australia}\\*[0.5mm]
{\em E-mail address}: {\tt sheehan.olver@sydney.edu.au}

\medskip

\noindent 
(Georg Wechslberger) {\sc Zentrum Mathematik -- M3, Technische Universität München, Germany}\\*[0.5mm]
{\em E-mail address}: {\tt wechslbe@ma.tum.de}

}

\end{document}